\documentclass[12pt]{amsart}
\usepackage{graphicx}
\usepackage{color,graphics,graphicx,shortvrb} 
\usepackage{epsfig}
\usepackage{amssymb,amsmath,amsfonts}
\usepackage{newlfont}
\usepackage{caption}
\usepackage{latexsym} 

\newtheorem{theo}{Theorem}

\newtheorem{lema}{Lemma}
\newtheorem{rema}{Remark}
\newtheorem{prop}{Proposition}

\def \d {\displaystyle}

\def \barr {\begin{array}{l}}
\def \ear {\end{array}}
\def \beq {\begin{equation}}
\def \eeq {\end{equation}}
\def \beqn {\begin{eqnarray}}
\def \eeqn {\end{eqnarray}}
\def \f {\end{document}}

\def \h {\hspace}
\usepackage{amsfonts}
\def\dfrac{\displaystyle\frac}

\newcommand{\field}[1]{\mathbb{#1}}
\newcommand{\R}{\field{R}}

\newcommand{\N}{\field{N}}
\begin{document}
\title[Delayed wave equation without displacement term]{\bf Asymptotic behavior of a delayed wave equation without displacement term}
\author{Ka\"{\i}s Ammari}
\address{UR Analysis and Control of PDEs, UR13ES64, Department of Mathematics, Faculty of Sciences of Monastir, University of Monastir, 5019 Monastir, Tunisia}
\email{kais.ammari@fsm.rnu.tn}

\author{Boumedi\`ene Chentouf}
\address{Kuwait University, Faculty of Science, Department of Mathematics, Safat 13060, Kuwait}
\email{chenboum@hotmail.com}

\begin{abstract} 
This paper is dedicated to the investigation of the asymptotic behavior of a delayed wave equation without the presence of any position term. First, it is shown that the problem is well-posed in the sense of semigroups theory. Thereafter, LaSalle's invariance principle is invoked in order to establish the asymptotic convergence for the solutions of the system to a stationary position which depends on the initial data. More importantly, without any geometric condition such as BLR condition \cite{ba} in the control zone, the logarithmic convergence is proved by using an interpolation inequality combined with a resolvent method.
\end{abstract}   


\maketitle
\tableofcontents

\thispagestyle{empty}

\section{Introduction}
Let $\Omega$ be an open bounded connected set of $\R ^n$ having a smooth boundary $\Gamma=\partial \Omega$ of class $C^2$. Given a partition $(\Gamma_0, \Gamma_1)$ of $\Gamma$, we will be concerned with the wave equation
\beq
y_{tt} (x,t) - \Delta y(x,t)=0, \; \hspace{3cm} \mbox{in} \;\; \Omega \times (0,\infty),
\label{1}
\eeq
as well as the following boundary and initial conditions
\begin{equation}
\left \lbrace
\begin{array}{ll}
\frac{\textstyle \partial y }{\textstyle \partial \nu } (x,t) = 0, & \mbox{on} \, \Gamma_0 \times (0,\infty),\\[2mm]
\frac{\textstyle \partial y }{\textstyle \partial \nu } (x,t) = -\alpha y_t (x,t)-\beta y_t (x,t-\tau), & \mbox{on} \, \Gamma_1 \times (0,\infty),\\[1mm]
y(x,0)=y_0(x), \; y_t(x,0)=z_0(x), & x \in \Omega,\\
y_t(x,t)=f(x,t),  & (x,t) \in \Gamma_1 \times (-\tau,0),
\label{2}
\end{array}
\right.
\end{equation}
in which $\alpha >0, \, \beta \in \R$, $\nu$ is the unit normal of $\Gamma$ pointing towards the exterior of $\Omega$. Additionally, it is supposed that $\Gamma_1$ is nonempty while $\Gamma_0$ may be empty.

It is well-known that the wave equation has been actively and continuously analyzed either qualitatively or numerically or both. This has led to hundreds of research papers and hence it is quasi-impossible to cite all of them. Notwithstanding, we name few such as \cite{ka3,ka4,ba,burq,c1,c2,c3,c4,h,l2,LR,li1,li2,ma,mo,qr,rt,r1,t} and  shall survey those bearing a resemblance to (\ref{1})-(\ref{2}). Indeed, the author in \cite{l1} has been mainly preoccupied with the asymptotic convergence of solutions of (\ref{1})-(\ref{2}) in the case when there is no delay term ($\beta=0$). Recently, a boundary delayed wave equation has been considered in \cite{np} but under the condition $y=0$ on $\Gamma_0$. Later, a non-standard energy norm has been provided in \cite{bc1}-\cite{bc4} in order to analyze the asymptotic behavior of solutions to the system (\ref{1})-(\ref{2}) where no delay arises.

Based on the above discussion, the present article places primary emphasis on the study of the asymptotic behavior of solutions of the system (\ref{1}). The main motivation of our work is to deal with a wave equation under the occurrence of a boundary delay term (in contrast to \cite{l1} and \cite{bc1}-\cite{bc4}) but at the same time when no position term appears in the system (contrary to \cite{np}). This will permit to extend the outcome in \cite{l1} and \cite{bc1}-\cite{bc4} in the sense that we do take into account the presence of a delay phenomenon. Additionally, we shall be able to provide a convergence result of solutions to (\ref{1})-(\ref{2}) despite the absence of any position term in contrast to \cite{np}.

The rest of this work is organized as follows. In Section 2, preliminaries are given and the problem is set up.  Section 3  is devoted to the proof of existence and uniqueness of solutions of our system. Section 4 deals with the asymptotic convergence of solutions to an equilibrium state. In Section 5, it is proved that the convergence is in fact polynomial. Finally, this note ends with a conclusion.

\section{Preliminaries and problem statement}
\setcounter{equation}{0}
This aim of this section is to set our problem in an appropriate functional space.  To proceed, we assume without loss of generality that $\beta > 0$. Thereafter, consider the standard change of state variable \cite{da}
$$u(x,\rho,t)=y_t(x,t-\tau \rho ), \, \quad \quad x \in \Gamma_1, \, \rho \in (0 ,1), \,   t>0,  $$ 
and the state space
$$ {\cal H}=H^1(\Omega) \times L^2(\Omega) \times L^2(\Gamma_1 \times (0,1)),$$
equipped with the inner product
\beq
\barr
\langle (y,z,u), (\tilde y,\tilde z, \tilde u) \rangle_{\scriptscriptstyle {\cal H}} = \d \int_{\Omega} \left( \nabla y \nabla \tilde y+z \tilde z \right)  dx + \xi \int_0^1 \int_{\Gamma_1} u \tilde u \, d\sigma d\rho  \; + \\
 \d \varpi \left[   \int_{\Omega} z \, dx  + (\alpha + \beta) \int_{\Gamma_1} y \, d\sigma -\beta \tau \int_0^1 \int_{\Gamma_1} u  \, d\sigma d\rho  \right] \left[  \int_{\Omega} \tilde z \, dx  + (\alpha + \beta) \int_{\Gamma_1}  \tilde y \, d\sigma -\beta \tau \int_0^1 \int_{\Gamma_1}  \tilde u \, d\sigma d\rho  \right], \label{5}
\ear
\eeq
where $\varpi >0$ is a positive constant to be determined. Additionally, $\alpha, \, \beta$ and $\xi$ satisfy 
\beq
\barr
0 < \beta < \alpha,\\
\tau \beta < \xi < \tau (2\alpha-\beta).
\label{xx}
\ear
\eeq
 
The result below addresses the issue of equivalence of the usual norm of our state space ${\cal H}=H^1(\Omega) \times L^2(\Omega) \times L^2(\Gamma_1 \times (0,1))$ and that induced by the inner product defined by (\ref{5}).
\begin{prop}
The state space ${\cal H}$ equipped with the inner product (\ref{5}) is a Hilbert space provided that $\varpi$ is 
small enough.
\label{p1}
\end{prop}

\begin{proof} 
The immediate task is to show the existence of two positive constants $L_1$ and $L_2$ such that
\beq L_1 \|(y, z,u) \|_{\scriptscriptstyle {H^1(\Omega) \times L^2(\Omega) \times L^2(\Gamma_1 \times (0,1))}} \leq \|(y, z,u) \|_{\scriptscriptstyle {\cal H}} \leq L_2 \|(y, z,u) \|_{\scriptscriptstyle {H^1(\Omega) \times L^2(\Omega) \times L^2(\Gamma_1 \times (0,1))}}.
\label{in}
\eeq
Invoking Young's and H\"older's inequalities, (\ref{5}) gives
\begin{eqnarray*}
\|(y, z,u) \|_{\scriptscriptstyle {\cal H}}^2 & \leq & \int_{\Omega}  |\nabla y |^2  dx + (1+3 \varpi {mes}(\Omega) ) \int_{\Omega} z^2 \, dx + (\xi +3 \beta^2 \tau^2 ) \int_0^1 \int_{\Gamma_1} u^2 \, d\sigma d\rho \\
&+& 3 \varpi (\alpha + \beta)^2 {mes}(\Gamma_1)    \int_{\Gamma_1}  y^2 \, d\sigma.
\end{eqnarray*}
Thereafter using the trace Theorem \cite{a}, the above inequality yields
\begin{eqnarray*}
\|(y, z,u) \|_{\scriptscriptstyle {\cal H}}^2 & \leq & ( 1+3 K \varpi (\alpha + \beta )^2 {mes}(\Gamma_1 ) ) \int_{\Omega}  |\nabla y |^2  dx +3 K \varpi (\alpha + \beta )^2 {mes}(\Gamma_1 )  \int_{\Omega}  y^2  dx \\
&+& ( 1+3 \varpi {mes}(\Omega) ) ) \int_{\Omega} z^2 \, dx + (\xi +3 \beta^2 \tau^2 ) \int_0^1 \int_{\Gamma_1} u^2 \, d\sigma d\rho,
\end{eqnarray*}
where $K$ is a positive constant depending on $\Omega$ \cite{a}. This leads to the first inequality in (\ref{in}). With regard to the second one, we use (\ref{5}) and apply Young's inequality to get
\begin{eqnarray}
\|(y, z,u) \|_{\scriptscriptstyle {\cal H}}^2 &=& \int_{\Omega} \left( |\nabla y |^2 + z^2 \right) dx + \varpi \left[   \int_{\Omega} z \, dx -\beta \tau   \int_0^1 \int_{\Gamma_1} u \, d\sigma d\rho \right]^2  + \varpi (\alpha + \beta )^2 \left[   \int_{\Gamma_1} y \, d\sigma \right]^2   \nonumber\\
&+& \xi \int_0^1 \int_{\Gamma_1} u^2 \, d\sigma d\rho + 2 \varpi (\alpha + \beta )\left[   \int_{\Gamma_1} y \, d\sigma \right]  \left[   \int_{\Omega} z \, dx -\beta \tau   \int_0^1 \int_{\Gamma_1} u \, d\sigma d\rho \right] \nonumber\\
&\geq& \int_{\Omega} \left( |\nabla y |^2 +z^2 \right) dx  + \xi \int_0^1 \int_{\Gamma_1} u^2 \, d\sigma d\rho +
\varpi (\alpha + \beta ) \left[ \alpha + \beta  -\delta \right] \left[   \int_{\Gamma_1} y \, d\sigma \right]^2  \nonumber\\
&+& \varpi \left[ 1-(\alpha + \beta)  \delta^{-1} \right] \left[   \int_{\Omega} z \, dx -\beta \tau   \int_0^1 \int_{\Gamma_1} u \, d\sigma d\rho \right]^2, \label{6}
\end{eqnarray}
for any positive constant $\delta$. This, together with the generalized Poincar\'e inequality
$$
\int_{\Omega}  y ^2  dx \leq C \left\{ \int_{\Omega}  |\nabla y|^2  dx + \left( \int_{\Gamma_1} y \, d\sigma \right)^2  \right \},$$
implies that
\begin{eqnarray}
\|(y, z,u) \|_{\scriptscriptstyle {\cal H}}^2 &\geq& \varpi (\alpha + \beta)(\alpha + \beta-\delta) C_2^{-1} \int_{\Omega}    y ^2 \, dx +\left[1-\varpi (\alpha + \beta)(\alpha + \beta-\delta) \right] \int_{\Omega} |\nabla y |^2 \, dx + \int_{\Omega} z^2 \, dx \nonumber\\
&+&  \varpi \left[ 1-(\alpha + \beta) \delta^{-1} \right] \left[   \int_{\Omega} z \, dx -\beta \tau   \int_0^1 \int_{\Gamma_1} u \, d\sigma d\rho \right]^2  + \xi \int_0^1 \int_{\Gamma_1} u^2 \, d\sigma d\rho , \label{9}
\end{eqnarray}
provided that $\delta$ satisfies $\delta < \alpha + \beta$. Note also that $C >0$ is the Poincar\'e constant which depends solely on $\Omega$. Applying again Young's and H\"older's inequalities, it follows from (\ref{9}) that for any $\delta < \alpha + \beta$
\begin{eqnarray}
\|(y, z,u) \|_{\scriptscriptstyle {\cal H}}^2 &\geq& \varpi (\alpha + \beta)(\alpha + \beta-\delta) C^{-1} \int_{\Omega}    y^2 \, dx +\left[1-\varpi (\alpha + \beta)(\alpha + \beta-\delta) \right] \int_{\Omega} |\nabla y |^2 \, dx  \nonumber\\
&+&  \left[ 1+ 2 \varpi (1-(\alpha + \beta ) \delta^{-1} ) {mes}(\Omega) \right]  \int_{\Omega} z^2 \, dx \nonumber\\
&+&  \left[ \xi +2 \varpi (1-(\alpha + \beta ) \delta^{-1} ) \beta^2 \tau^2 {mes}(\Gamma_1 ) \right]   \int_0^1 \int_{\Gamma_1} u^2 \, d\sigma d\rho. \label{9v}
\end{eqnarray}
Finally, we pick up $\varpi$ such that 
$$
\varpi < \min \left \{ \frac{1}{(\alpha + \beta)(\alpha + \beta-\delta)}, \frac{\delta}{2(\alpha + \beta-\delta) {mes}(\Omega)} , \frac{\delta \xi}{2(\alpha + \beta-\delta) {mes}(\Gamma_1)} \right\}.
$$
which gives rise to  the second inequality of (\ref{in}).

\end{proof}

Now we are ready to formulate the system (\ref{1})-(\ref{2}) in an abstract differential equation in the Hilbert state space $ {\cal H}$ equipped with the inner product (\ref{5}). First, let us recall that $u(x,\rho,t)=y_t(x,t-\tau \rho ), \,\, x \in \Gamma_1, \, \rho \in (0 ,1), \,   t>0  $ and let $z=y_t$, $\Phi =(y, z,u)$. Whereupon, the closed loop system can be written as follows
\beq
\left\{
\barr
\Phi_t (t)= {\cal A} \Phi (t),\\
\Phi (0)= \Phi_0 =(y_0,z_0,f),
\ear
\right.
\label{s}
\eeq
where ${\cal A}$ is an unbounded linear operator defined by
\beq
\barr
 {\cal D}({\cal A}) =\Bigl\{ (y, z,u ) \in H^1(\Omega) \times H^1(\Omega) \times L^2(\Gamma_1 \times (0,1)); \Delta y \in L^2(\Omega); \,
\frac{\textstyle \partial y }{\textstyle \partial \nu } = 0 \, \mbox{on} \, \Gamma_0; \\
\hspace{5.2cm} \frac{\textstyle \partial y }{\textstyle \partial \nu }+\alpha z + \beta u(\cdot,1)=0, \, \mbox{and} \, \; z=u(\cdot,0) \, \mbox{on} \, \Gamma_1
\label{10}
\Bigr \},
\ear
\eeq
and  
\beq
{\cal A} (y, z,u ) = ( z, \Delta y, -\tau^{-1} u_{\rho} ), \quad  \forall (y, z,u ) \in {\cal D}({\cal A}).
\label{11}
\eeq

\section{Well-posedness of the problem}
This section addresses the problem of existence and uniqueness of solutions of the system (\ref{s}) in ${\cal H}$. To do so, we shall evoke semigroups theory. We have

\begin{prop} Assume that the conditions (\ref{xx}) hold. Then, the linear operator ${\cal A}$ generates a $C_0$ semigroup of contractions $S(t)$ on ${\cal H}=\overline {{\cal D}({\cal A})}$. Additionally, for any initial data $\Phi_0 \in {\cal D}({\cal A})$, the system (\ref{s}) possesses a unique strong solution $\Phi (t) S(t) \Phi_0 \in {\cal D}({\cal A})$ for all $ t \geq 0$ such that
$ \Phi (\cdot)  \in C^1 (\R^+;{\cal H}) \cap C  ( \R^+;{\cal D} ({\cal A}))$. In turn, if $\Phi_0 \in {\cal H}$, then the system (\ref{s}) has a unique weak solution $\Phi (t)= S(t) \Phi_0 \in {\cal H}$ such that $ \Phi (\cdot)  \in C^0 (\R^+;{\cal H} )$.
\label{l1}
\end{prop}

\begin{proof}
The ultimate outcome will be the proof of the dissipativity and maximality of the operator ${\cal A}$. First of all, let $\Phi=(y, z,u ) \in {\cal D}({\cal A})$. Using (\ref{5}) and (\ref{11}), we obtain 
\beq
\barr
\langle {\cal A} \Phi, \Phi \rangle_{\scriptscriptstyle {\cal H}} = 
\d \int_{\Omega} \left( \nabla y \nabla z+\Delta y z \right)  dx - \xi \tau^{-1} \int_0^1 \int_{\Gamma_1} u_{\rho} u \, d\sigma d\rho +  \\
 \d \varpi \left[   \int_{\Omega} \Delta \, dx  + (\alpha + \beta) \int_{\Gamma_1} z \, d\sigma +\beta \tau \int_0^1 \int_{\Gamma_1} u_{\rho}  \, d\sigma d\rho  \right] \left[  \int_{\Omega} z \, dx  + (\alpha + \beta) \int_{\Gamma_1} y \, d\sigma -\beta \tau \int_0^1 \int_{\Gamma_1}  u \, d\sigma d\rho  \right]. \label{dis1}
\ear
\eeq
This implies, thanks to Green formula and (\ref{10}), that
\beq 
\left< {\cal A} \Phi, \Phi \right>_{\scriptscriptstyle {\cal H}}=-\alpha \int_{\Gamma_1} z^2 \, d\sigma -\beta \int_{\Gamma_1} z u(x,1) \, d\sigma -\xi \tau^{-1} \int_0^1 \int_{\Gamma_1} u_{\rho} u \, d\sigma d\rho. \label{ps} 
\eeq
Applying Young's inequality and using the fact that $2 \int_0^1 \int_{\Gamma_1} u_{\rho} u \, d\sigma d\rho = \int_{\Gamma_1} ( u^2 (\sigma,1)- u^2 (\sigma,0) ) \, d\sigma=\int_{\Gamma_1} ( u^2 (\sigma,1)- z^2 ) \, d\sigma$, we deduce from (\ref{ps}) that
\beq 
\left< {\cal A} \Phi, \Phi \right>_{\scriptscriptstyle {\cal H}}\leq \dfrac{1}{2} \left( \beta-2\alpha+ \xi \tau^{-1}\right) \int_{\Gamma_1} z^2 \, d\sigma +\dfrac{1}{2} \left( \beta-\xi \tau^{-1}\right) \int_{\Gamma_1} u^2(\sigma,1) \, d\sigma. \label{diss} 
\eeq
Exploring the assumptions (\ref{xx}), one can claim that the operator ${\cal A}$ is dissipative.  

Now, let us show that the operator $( I-{\cal A})$ is onto, which is equivalent to prove that given $(f,g,v) \in {\cal H}$, we seek $(y,z,u ) \in {\cal D}({\cal A})$ such that 
 $(I-{\cal A}) (y,z,u)=(f,g,v)$, that is, 
\beq \left \lbrace \begin {array} {ll} 
y-z=f,& \mbox{in} \, \Omega \\ 
z-\Delta y =g, & \mbox{in} \, \Omega \\ 
u_{\rho} =-\tau u + \tau v, & \mbox{on} \, \Gamma_1 \times (0,1),\\
\d \frac{\textstyle \partial y }{\textstyle \partial \nu } = 0, & \mbox{on} \, \Gamma_0,\\[2mm]
u(\cdot,0)=z, & \mbox{on} \, \Gamma_1,\\[1mm] 
\d \frac{\textstyle \partial y }{\textstyle \partial \nu } +\alpha z + \beta u(\cdot,1)=0, & \mbox{on} \, \Gamma_1. \end{array} \right. \label{max}
\eeq 
Whereupon, $z=y-f$. Moreover, solving the equation of $u$ in the above system and recalling that $u(\cdot,0)=z$ on $\Gamma_1$, we get
\begin{equation}
\displaystyle u(x,\rho) = e^{-\tau \rho} y(x) -e^{-\tau \rho} f(x) \displaystyle +\tau \int_0 ^{\rho} e^{\tau (\eta- \rho)} v(x,\eta) \, d \eta.
\label{uu1} 
\end{equation}   
Thus
$$\displaystyle u(x,1) = e^{-\tau } y(x) +v_f(x),
$$
where
$ v_f (x)= -e^{-\tau} f(x) \displaystyle +\tau \int_0 ^{1} e^{\tau (\eta-1)} v(x,\eta) \, d \eta.$ In the light of the above arguments, one has only to seek $y \in H^2(\Omega)$ satisfying 
\beq \left \lbrace \begin {array} {ll} 
y^2-\Delta y =f+g, & \mbox{in} \, \Omega \\ 
\d \frac{\textstyle \partial y }{\textstyle \partial \nu } = 0, & \mbox{on} \, \Gamma_0,\\[2mm]
\d \frac{\textstyle \partial y }{\textstyle \partial \nu } +\left( \alpha +\beta e^{-\tau} \right) y-\alpha f + \beta v_f=0, & \mbox{on} \, \Gamma_1. \end{array} \right. \label{max2}
\eeq 
Using Green formula, one can prove that the system (\ref{max2}) is equivalent to the following variational equation 
\beq \begin
 {array} {l} \d \h{3mm} \int_{\Omega} \biggl( y \phi +\nabla y \nabla \phi \biggr)\, dx+\int_{\Gamma_1} \left( \alpha +\beta e^{-\tau} \right) y \phi \, d\sigma= \int_{\Omega} \left(f+g\right) \phi \, dx+\int_{\Gamma_1} (\alpha f - \beta v_f) \phi \, d\sigma, \end{array} 
\label{n4n} 
\eeq 
for any $\phi \in H^1(\Omega)$. Invoking Lax-Milgram Theorem \cite{br}, one can prove that (\ref{n4n}) admits a unique solution $y \in H^1(\Omega)$. Then, thanks to standard arguments used for solving elliptic linear equations, one can  recover the boundary conditions in (\ref{max2}). Herewith, the operator $I-{\cal A}$ is onto. As a direct consequence of Lummer-Phillips theorem \cite{p}, the  operator ${\cal A}$ is densely defined closed in ${{\cal H}}$ and generates a $C_0$-semigroup of contractions  $S (t)$ on ${\cal H}$. Lastly, the rest of the claims in Proposition \ref{l1} follows from semigroups theory \cite{p} (see also \cite{br}).
\end{proof}

\section{Asymptotic behavior}
\setcounter{equation}{0}
In this section, we will establish an asymptotic behavior result for the unique solution of (\ref{s}) in ${\cal H}$. 
The first main result of this work is:
\begin{theo} Assume that the conditions (\ref{xx}) hold. Then, for any initial data $\Phi_0=(y_0,z_0,f) \in {\cal H}$, the solution $\Phi (t)=(y(\cdot,t), y_t (\cdot,t)),y_t (\cdot,t-\tau \rho))$ of (\ref{s}) tends in ${\cal {H}}$ to $(\chi ,0,0)$ as $t \longrightarrow +\infty $, where
$$\chi =\displaystyle  \left( (\alpha+\beta) {mes}(\Gamma_1) \right)^{-1} \left( \int_{\Omega} z_0 \, dx + (\alpha+\beta) \int_{\Gamma_1}  y_0 \, d\sigma -\beta \tau \int_0^1 \int_{\Gamma_1} f \, d\sigma d\rho  \right).$$
\label{t1}
\end{theo}
\begin{proof}
By a standard argument of density of ${\cal D}({\cal A}^2)$ in ${\cal H} $
and the contraction of the semigroup $S(t)$, it suffices to prove Theorem \ref{t1} for smooth initial data $ \Phi_0  \in {\cal D}({\cal A}^2)$. Thereby, let $\Phi(t)=S(t) \Phi_0$ be the solution of (\ref{s}). It follows from Lemma \ref{l1} that the trajectory of solution $\{ \Phi(t) \}_{\scriptscriptstyle t \geq0}$ is a bounded set for the graph norm and thus precompact. Applying LaSalle's principle, we deduce that $\omega \left( \Phi_0 \right) $ is non empty, compact, invariant under the semigroup $S(t)$ and in addition $S(t) \Phi_0  \longrightarrow \omega \left( \Phi_0 \right)  \; $ as $ t \to  +\infty \, $ \cite{h}. Thenceforth, it suffices to show that $ \omega \left( \Phi_0 \right)$ reduces to $(\chi,0,0)$.  To this end, let
 $\tilde \Phi_0=\left(\tilde y_0,\tilde z_0, \tilde u_0\right) \in \omega \left( \Phi_0 \right) \subset {\cal D}({\cal A})$ and consider the unique strong solution of (\ref{1})-(\ref{2}) $\tilde \Phi(t)=\left(\tilde y (t),\tilde y_t (t),\tilde y_t(\cdot,t-\tau \rho )  \right)=S(t) \tilde \Phi_0 \in {\cal
 D}({\cal A})$. In view of the fact that $ \|\tilde \Phi (t)\|_{\scriptscriptstyle {\cal H}}$ is constant \cite{h}, we have $<{\cal A} \tilde \Phi,\tilde \Phi>_{\scriptscriptstyle {\cal H}}=0.$ Putting the above deductions together with (\ref{diss}) yields $\tilde z=\tilde y_t =0$ and $\tilde u(x,1)=\tilde y_t (x,t-\tau)=0$ on $\Gamma_1$. Whence $\tilde y$ satisfies the following 
\beq \left \lbrace \begin{array}{ll} 
\tilde y_{tt} - \Delta \tilde y = 0, & \mbox{in} \, \Omega \\
\frac{\textstyle \partial \tilde y }{\textstyle \partial \nu }= 0, & \mbox{on} \, \Gamma_0,\\[1mm]
\tilde y_{t}=\frac{\textstyle \partial \tilde y }{\textstyle \partial \nu } = 0, & \mbox{on} \, \Gamma_1,\\[1mm] 
\tilde y(0)=\tilde y_0; \, \tilde y_t (0)=\tilde z_0, & \mbox{in} \, \Omega, \\ 
\tilde y \in H^2(\Omega), \end{array} \right. \label{e2n} 
\eeq 
which in turn implies that $\tilde z=\tilde y_t$ is solution of 
\beq \left \lbrace \begin{array}{ll} 
\tilde z_{tt} -\Delta \tilde z = 0, & \mbox{in} \, \Omega \\ 
\tilde z =\frac{\textstyle \partial \tilde z}{\textstyle \partial \nu } = 0, & \mbox{on} \, \Gamma_0,\\[2mm] 
\tilde z =\frac{\textstyle \partial \tilde z}{\textstyle \partial \nu } = 0, & \mbox{on} \, \Gamma_1. 
\end{array} \right.
 \label{e3n} \eeq 
Evoking Holmgren's uniqueness theorem for the system (\ref{e3n}), one can claim that $\tilde
 z = 0$ and thus the system (\ref{e2n}) becomes 
$$ \left \lbrace \begin{array}{ll} 
\Delta  \tilde y = 0, & \mbox{in} \, \Omega \\ 
\frac{\textstyle \partial \tilde y }{\textstyle \partial \nu }= 0, & \mbox{on} \, \Gamma_0,\\[2mm] 
\frac{\textstyle \partial \tilde y }{\textstyle \partial \nu } = 0, & \mbox{on} \, \Gamma_1. 
\end{array} \right. $$ 
Consequently, $\tilde y$ is constant. To summarize, we have proved that given $\tilde \Phi_0=(\tilde y_0,\tilde z_0,\tilde u_0 ) \in \omega \left( \Phi_0 \right) \subset {\cal D}({\cal A})$, the solution $\tilde \Phi(t)=(\tilde
 y (t),\tilde y_t (t), \tilde y_t (\cdot,t- \rho \tau) )=S(t) \tilde \Phi_0 \in {\cal D}({\cal A})$ satisfies $(\tilde y (t),\tilde y_t (t), \tilde y_t (\cdot,t- \rho \tau) )=({\chi},0,0)$, for any $t \geq 0$, where $\chi$ is a real constant. Henceforth, the $\omega$-limit set $\omega \left( \Phi_0 \right)$ consists of constants $(\chi,0)$. It remains now  to provide an explicit form of $\chi$. To proceed, assume that $({\chi},0,0) \in \omega \left( \Phi_0 \right)$. This implies that there exists $\left \{ t_n \right \} \to \infty$,
 as $n \to \infty$ such that 
\beq \Phi(t_n)=(y(t_n), y_t (t_n), y_t (\cdot,t_n-\rho \tau) )=S(t_n) \Phi_0 \longrightarrow ({\chi},0,0), \quad \mbox{as} \, n \to \infty,   
\label{boum} \eeq 
in $ {\cal H}$. On the other hand, we claim that any solution of the system (\ref{1})-(\ref{2}) obeys the following property 
$$  {\cal E}(t)= \int_{\Omega} y_t \, dx  + (\alpha + \beta) \int_{\Gamma_1} y \, d\sigma -\beta \tau \int_0^1 \int_{\Gamma_1} y_t (x,t- \rho \tau)  \, d\sigma d\rho $$
is time-invariant. To ascertain the correctness of this claim, let us differentiate the above expression with respect to $t$ and then use (\ref{1})-(\ref{2}) together with Green formula. A straightforward computation gives 
\begin{eqnarray*}
{\cal E}_t(t)&=& -\int_{\Gamma_1} (\alpha y_t +\beta y_t(x,t- \tau)  \, d\sigma + (\alpha + \beta) \int_{\Gamma_1} y_t \, d\sigma- \beta \tau \int_0^1 \int_{\Gamma_1} y_{tt} (x,t- \rho \tau)  \, d\sigma d\rho\\
&=& -\beta \int_{\Gamma_1}  y_t(x,t- \tau)  \, d\sigma + \beta \int_{\Gamma_1} y_t \, d\sigma - \beta \tau \int_0^1 \int_{\Gamma_1} \tau^{-2} y_{\rho \rho} (x,t- \rho \tau)  \, d\sigma d\rho \\
&=& -\beta \int_{\Gamma_1}  y_t(x,t- \tau)  \, d\sigma + \beta \int_{\Gamma_1} y_t \, d\sigma - \beta \tau^{-1} \int_{\Gamma_1} ( y_{\rho} (x,t-\tau)-y_{\rho} (x,t) )  \, d\sigma \\
&=& -\beta \int_{\Gamma_1}  y_t(x,t- \tau)  \, d\sigma + \beta \int_{\Gamma_1} y_t \, d\sigma - \beta \tau^{-1} \int_{\Gamma_1} (-\tau)( y_{t} (x,t-\tau)-y_{t} (x,t) )  \, d\sigma =0,
\end{eqnarray*}
which confirms our claim. This yields 
\begin{equation} \int_{\Omega} y_t \, dx  + (\alpha + \beta) \int_{\Gamma_1} y \, d\sigma -\beta \tau \int_0^1 \int_{\Gamma_1} y_t (x,t-\rho \tau)  \, d\sigma d\rho = \int_{\Omega} z_0 \, dx  + (\alpha + \beta) \int_{\Gamma_1} y_0 \, d\sigma -\beta \tau \int_0^1 \int_{\Gamma_1} f  \, d\sigma d\rho. 
\label{bou} 
\end{equation} 
It suffices now to pick $t=t_n$ in (\ref{bou}) and let $n \to \infty$. This, together with (\ref{boum}), implies that   
$$ 0  + (\alpha + \beta) \int_{\Gamma_1} \chi \, d\sigma -0=\displaystyle \int_{\Omega} z_0 \, dx  + (\alpha + \beta) \int_{\Gamma_1} y_0 \, d\sigma -\beta \tau \int_0^1 \int_{\Gamma_1} f  \, d\sigma d\rho $$
and hence
$${ \chi}=\displaystyle  \left( (\alpha+\beta) {mes}(\Gamma_1) \right)^{-1} \left( \int_{\Omega} z_0 \, dx + (\alpha+\beta) \int_{\Gamma_1}  y_0 \, d\sigma -\beta \tau \int_0^1 \int_{\Gamma_1} f \, d\sigma d\rho  \right).$$

\end{proof}

\section{Logarithmic convergence}
As it has been proved in the previous section that the solutions of our closed-loop system asymptotically converge to an equilibrium state, it is legitimate to wonder how is that convergence. In this section, we will mainly be concerned with the answer of such a question. In fact, we shall show that the convergence is actually logarithmic.

Let $\dot{\mathcal{H}}$ the closed subspace of $\mathcal{H}$ and of codimension $1$ given by
$$
\dot{\mathcal{H}} = \left\{(y,z,u) \in \mathcal{H}; \, 
\int_\Omega z(x) \, dx - \beta \tau \, \int_0^1 \int_{\Gamma_1} u(\sigma,\rho) \, d \sigma \, d \rho + (\alpha + \beta) \, \int_{\Gamma_1} y \, d \sigma = 0  \right\}
$$
and denote by $\dot{\mathcal{A}}$ a new operator defined as follows
$$
\dot{\mathcal{A}} : \mathcal{D}(\dot{\mathcal{A}}) := \mathcal{D}(\mathcal{A}) \cap \dot{\mathcal{H}} \subset \dot{\mathcal{H}} \rightarrow \dot{\mathcal{H}},
$$

\begin{equation}
\label{1.62bis}
\dot{\mathcal{A}} (y,z,u) = \mathcal{A} (y,z,u), \, \forall \, (y,z,u) \in \mathcal{D}(\dot{\mathcal{A}}).
\end{equation}

Assume that the conditions (\ref{xx}) hold. Then, we have, thanks to results of previous sections, that the operator $\dot{\mathcal{A}}$ defined by (\ref{11}) generates on $\dot{{\mathcal H}}$ a $C_0$-semigroup of contractions $e^{t \dot{\mathcal{A}}}$. Moreover, $\sigma(\dot{\mathcal{A}})$, the spectrum of $\dot{\mathcal{A}}$, consists of isolated eigenvalues of finite algebraic multiplicity only.

More importantly, the semigroup operator $e^{t \dot{\mathcal{A}}}$ is logarithmic stable on $\dot{\mathcal{H}}$. Indeed, our second main result is:
\begin{theo} \label{lrkv}
Suppose that the assumptions (\ref{xx}) are satisfied. Then, there exists $C >0$ such that for all $t > 0$ we have 
$$
\left\|e^{t \dot{\mathcal{A}}}\right\|_{{\mathcal L} (\mathcal{D}(\dot{\mathcal{A}}),\dot{\mathcal{H}})} \leq \frac{C}{\log{(2 + t)}}.
$$ 
\end{theo}
 
\begin{proof} 

We will employ the following frequency domain theorem for logarithmic stability from \cite[Theorem A]{BD} (see also \cite{burq}) of a $C_0$ semigroup of contractions on a Hilbert space:

\begin{lema}
\label{lemraokv}
A $C_0$ semigroup $e^{t{\mathcal L}}$ of contractions on a Hilbert space ${\mathcal X}$ satisfies, for all $t >0$, 
$$||e^{t{\mathcal L}}||_{{\mathcal L}(\mathcal{D}(\mathcal{A}),{\mathcal X})} \leq \frac{C}{\log{(2+t)}}$$
for some constant $C >0$  if 
\begin{equation}
\rho ({\mathcal L})\supset \bigr\{i \gamma \bigm|\gamma \in \R \bigr\} \equiv i \R, \label{1.8wkv} \end{equation}
and
 \begin{equation}\limsup_{|\gamma| \to \infty }  \| e^{-c |\gamma|} \, (i\gamma I -{\mathcal L})^{-1}\|_{{\mathcal L}({\mathcal X})} <\infty, \; \hbox{for some} \;  c > 0,\label{1.9kv} 
\end{equation}
where $\rho({\mathcal L})$ denotes the resolvent set of the operator 
${\mathcal L}$.
\end{lema}

The proof of Theorem \ref{lrkv} is based on the following lemmas.

\medskip

We first look at the point spectrum.

\begin{lema}\label{condspkv}
If $\gamma$ is a real number, then $i\gamma$  is not an eigenvalue of $\dot{\mathcal{A}}$.
\end{lema}
\begin{proof} We
will show that the equation 
\begin{equation} 
\dot{\mathcal{A}} Z = i \gamma Z 
\label{eigenkv} 
\end{equation}
with $Z= (y,z,u)^T \in \mathcal{D}(\dot{\mathcal{A}})$ and $\gamma \in \mathbb{R}$ has only
the trivial solution.

\medskip

Clearly, the system \eqref{eigenkv} writes:
\begin{eqnarray}
&&z = i \gamma y \label{eigen1}\\
&& \Delta y  = i \gamma z, \label{eigen2}\\
&&- \frac{u_\rho}{\tau}  = i \gamma u, \label{eigen3}\\
&&\int_\Omega z(x) \, dx - \beta \tau \, \int_0^1 \int_{\Gamma_1} u(\sigma,\rho) \, d \sigma \, d \rho + (\alpha + \beta) \, \int_{\Gamma_1} y \, d \sigma = 0, \label{eigen4}\\
&& \frac{\textstyle \partial y }{\textstyle \partial \nu } = 0 \,\,\, \mbox{on} \, \Gamma_0, \label{eigen5}\\
&&\frac{\textstyle \partial y }{\textstyle \partial \nu }+\alpha z + \beta u(\cdot,1)=0, \, \mbox{and} \, \; z=u(\cdot,0) \,\,\, \mbox{on} \, \Gamma_1. 
\label{eigen6}
\end{eqnarray}
\noindent 
Let us firstly deal with the case where $\gamma = 0$. In such an event, It is clear, thanks to (\ref{eigen1})-(\ref{eigen6}), that the only solution of \eqref{eigenkv} is the trivial one.

\medskip

Let us suppose now that $\gamma \neq 0$. By taking the inner product of (\ref{eigenkv}) with $Z$,
using  the inequality (\ref{diss}) we get: 
\begin{equation}\label{1.7kv}
\Re \left(<\dot{\mathcal{A}} Z,Z>_{{\dot{\mathcal{H}}}} \right) \leq \frac{1}{2} \, \left( (\beta - 2\alpha + \xi \tau^{-1} ) \int_{\Gamma_1} |z(\sigma)|^2 \, \d\sigma + (\beta -\xi \tau^{-1}) \, \int_{\Gamma_1}  |u(\sigma,1)|^2 \, d\sigma \right) (\leq 0).
\end{equation}
Whereupon, we have $z=u(\cdot,1) = 0$ on $\Gamma_1$. Consequently, the only solution of \eqref{eigenkv} is the trivial one.
\end{proof}

\begin{lema}\label{lemresolventkv}
The resolvent operator of $\dot{\mathcal{A}}$ satisfies condition \eqref{1.9kv}.
\end{lema}

\begin{proof}
Suppose that condition \eqref{1.9kv} is false. 
By Banach-Steinhaus Theorem (see \cite{br}), there exist a sequence of real numbers $\gamma_n \rightarrow \infty$ and a sequence of vectors 
$Z_n= (y_n,z_n,u_n)^T \in \mathcal{D}(\dot{\mathcal{A}})$ with $\|Z_n\|_{\dot{\mathcal{H}}} = 1$ such that 
\begin{equation}
\| e^{c|\gamma_n|} \, (i \gamma_n I - \dot{\mathcal{A}})Z_n\|_{\dot{\mathcal{H}}} \rightarrow 0\;\;\;\; \mbox{as}\;\;\;n\rightarrow \infty, 
\label{1.12kv} \end{equation}
i.e., 
\begin{equation} e^{c|\gamma_n|} \left(i \gamma_n y_n - z_n \right)  \equiv e^{c|\gamma_n|} \, f_{n}\rightarrow 0 \;\;\; \mbox{in}\;\; H^1(\Omega), 
\label{1.13kv}\end{equation}
 \begin{equation}
  e^{c|\gamma_n|} \left( i \gamma_n
z_n - \Delta y_n \right)  \equiv  e^{c|\gamma_n|} \, g_n \rightarrow 0 \;\;\;
\mbox{in}\;\; L^2(\Omega),
\label{1.13bkv} \end{equation}
 \begin{equation}
 e^{c|\gamma_n|} \, \left( i \gamma_n u_{n} + \frac{(u_n)_\rho}{\tau} \right) \equiv  e^{c|\gamma_n|} \, v_{n} \rightarrow 0 \;\;\;
\mbox{in}\;\; L^2(\Gamma_1;(0,1)).
\label{1.14bkv} \end{equation}

The task ahead is to derive from \eqref{1.12kv} that $\|Z_n\|_{\dot{\mathcal{H}}}$ converges to zero. Obviously, such a result contradicts the fact that $ \forall \ n  \in \N \,\, \left\|Z_n\right\|_{\dot{\mathcal{H}}}=1.$ For sake of clarity, we shall prove our result by proceeding by steps.
\medskip

\begin{itemize}
\item
{\bf First step.}

\medskip

We first notice that we have
\begin{equation}
|| e^{c|\gamma_n|} \, (i \gamma_n I - \dot{\mathcal{A}})Z_n||_{\dot{\mathcal{H}}} \ge \left|\Re \left(\langle e^{c|\gamma_n|} \, (i\beta_n I - \dot{\mathcal{A}})Z_n, Z_n\rangle_{\dot{\mathcal{H}}} \right) \right| \ . 
\label{1.15kv}
\end{equation}
Thus by \eqref{1.7kv} and \eqref{1.12kv}, 
\begin{equation}
\label{cvz}
e^{\frac{c|\gamma_n|}{2}} \, z_{n} \rightarrow 0 \; \hbox{in} \; L^2(\Gamma_1), \, e^{\frac{c|\gamma_n|}{2}} \, u_n (\cdot,1) \rightarrow 0 \; \hbox{in} \; L^2(\Gamma_1),
\end{equation}
and hence
\begin{equation}
\label{cvzz}
u_n (\cdot,0) \rightarrow 0 \; \hbox{in} \; L^2(\Gamma_1).
\end{equation}
Moreover we have according to (\ref{1.13kv}) and (\ref{cvz}) that 
\begin{equation}
\label{cvyt}
i \gamma_n \, e^{\frac{c|\gamma_n|}{2}} \, y_n = e^{\frac{c|\gamma_n|}{2}} \, z_n + e^{\frac{c|\gamma_n|}{2}} \, f_n \rightarrow 0 \; \hbox{in} \; L^2(\Gamma_1),
\end{equation}
which implies that
\begin{equation}
\label{cvy}
e^{\frac{c|\gamma_n|}{2}} \, y_n \rightarrow 0 
\; \hbox{in} \; L^2(\Gamma_1).
\end{equation}

\medskip

Solving \eqref{1.13bkv} , it follows that 
$$
u_n(x,\rho) = u_n(x,0) \, e^{- i \tau \gamma_n x} + \tau \, \int_0^\rho e^{- i \tau \gamma_n (\rho - s)} \, v_n (s) \, ds.
$$
This, together with \eqref{1.13bkv} and \eqref{cvzz}  yield
\begin{equation} 
\label{znkv}
u_n \rightarrow 0 \; \; \mbox{in}\;\; L^2(\Gamma_1,(0,1))\ .
\end{equation}
 
\item
{\bf Second step.}

\medskip

We take the inner product of (\ref{1.13kv}) with $z_n = i \gamma_n y_n - f_n$ in $L^2(\Omega)$. We obtain that

$$
\int_\Omega |z_n|^2 \, dx - \int_\Omega |\nabla y_n |^2 \, dx  = 
$$
\begin{equation}
\label{eqzy}
- \int_{\Gamma_1} 
\frac{\partial y_n}{\partial \nu} \, \overline{y} \, d \sigma  + \frac{1}{i \gamma_n} \int_\Omega \nabla y_n \, \nabla \overline{f}_n \, dx   -
\frac{1}{i \gamma_n} \, \int_{\Gamma_1} \frac{\partial y_n}{\partial \nu} \, \overline{f}_n \, d\sigma + \frac{1}{i \gamma_n} \, \int_\Omega g_n \, \overline{z}_n \, dx  = \circ (1).
\end{equation}
Moreover according to the interpolation inequality given in \cite[Theorem 3]{LR}, there exists a constant $C >0$ such that  for sufficiently large $n$ we have:
$$
\left\|y_n\right\|_{H^1(\Omega)} \leq C \, e^{\frac{c|\gamma_n|}{2}} \, \left( \left\| f_n\right\|_{H^1 (\Omega)} + \left\|g_n \right\|_{L^2(\Omega)} + \left\|y_n\right\|_{L^2(\Gamma_1)} \right).
$$

Then \eqref{1.13kv}, \eqref{1.13bkv} and \eqref{cvy} gives
\begin{equation}
\label{cvyg}
y_n \rightarrow 0 \; \hbox{in} \; H^1(\Omega).
\end{equation}
Combining \eqref{eqzy} and \eqref{cvyg}, we get 
\begin{equation}
\label{zcvd}
z_n \rightarrow 0 \; \hbox{in} \; L^2(\Omega).
\end{equation}
Lastly, amalgamating \eqref{znkv}, \eqref{cvyg}  and  \eqref{zcvd}, we  clearly reach the contradiction with $ \forall \ n  \in \N \,,\, \left\|Z_n\right\|_{\dot{\mathcal{H}}}=1. $
\end{itemize}

\end{proof}

The two hypotheses of Lemma \ref{lemraokv} are proved. Thereby, the proof of Theorem \ref{lrkv} is achieved.

\end{proof}

\begin{rema}
It is worth mentioning that one may consider dynamical boundary conditions, that is, instead of (\ref{2}), one can treat 
\begin{equation} 
\left \lbrace \begin{array}{ll} 
m(x) y_{tt} (x,t) +\frac{\textstyle \partial y }{\textstyle \partial \nu }  (x,t)= 0, & (x,t) \in \Gamma_0 \times (0,\infty)\\[2mm] 
M(x) y_{tt} (x,t) +\frac{\textstyle \partial y }{\textstyle \partial \nu } = -\alpha y_t (x,t)-\beta y_t (x,t-\tau), & (x,t) \in \Gamma_1 \times (0,\infty)\\[1mm]
y(x,0)=y_0(x) \in H^1(\Omega), \; y_t(x,0)=z_0(x) \in L^2(\Omega), \\[1mm]
{y_t}|_{\scriptscriptstyle \Gamma_0} (x,0)=w_0^0 (x) \in L^2(\Gamma_0) ,\; {y_t}|_{\scriptscriptstyle \Gamma_1} (x,0)=w_1^0 (x) \in L^2(\Gamma_1),\\
y_t(x,t)=f(x,t),  & (x,t) \in \Gamma_1 \times (-\tau,0), \label{2n} \end{array} \right. 
\end{equation}  
in which 
\begin{equation} 
\left \lbrace \begin{array}{l} 
m \in L^{\infty} (\Gamma_0); \; m(x) \geq m_0 >0, \forall x \in \Gamma_0;\\[2mm] 
M \in L^{\infty} (\Gamma_1); \; M(x)
 \geq M_1 >0, \forall x \in \Gamma_1. \label{mmn} 
\end{array} \right. 
\end{equation}

Note that in this case, most of arguments used run on much the same lines as we did previously but of course with a number of changes born out of necessity. For instance, the new state space is 
$$ {\cal H}_d=H^1(\Omega) \times L^2(\Omega) \times L^2(\Gamma_1 \times (0,1)) \times L^2(\Gamma_0) \times L^2(\Gamma_1),$$ equipped with the inner product 
\beq 
\barr 
\biggl\langle (y,z,u,w_0,w_1), (\tilde y,\tilde z,\tilde u, \tilde w_0,\tilde w_1) \biggr\rangle_{\scriptscriptstyle {\cal H}_d} = \d \int_{\Omega} \left( y  \tilde y+z \tilde z \right) dx + \xi \int_0^1 \int_{\Gamma_1} u \tilde u \, d\sigma d\rho + \int_{\Gamma_0}  m w_0 \tilde w_0 \, d\sigma+ \\
\d \int_{\Gamma_1} M w_1 \tilde w_1 \, d\sigma + \varpi \left( \int_{\Omega} z \, dx +\int_{\Gamma_0} m w_0 \, d\sigma+ \int_{\Gamma_1} \left(Mw_1+ (\alpha + \beta) y \right) \, d\sigma -\beta \tau \int_0^1 \int_{\Gamma_1}  \tilde u \, d\sigma d\rho\right) \\
\hspace{3.5cm} \d \times \left( \int_{\Omega} \tilde z \, dx + \int_{\Gamma_0}
 m \tilde w_0 \, d\sigma + \int_{\Gamma_1} \left(M \tilde w_1+ (\alpha + \beta) \tilde y \right) \, d\sigma -\beta \tau \int_0^1 \int_{\Gamma_1}  \tilde u \, d\sigma d\rho \right).
 \label{5n}
\ear 
\eeq 
Thereafter, arguing as in the proof of Proposition \ref{p1} one can show that the above inner product generates an equivalent norm to the standard one provided that $\varpi >0$ is small enough. Additionally, the novel system operator has just two extra components $w_0$ and $w_1$ compared to ${\cal A}$ defined in (\ref{10})-(\ref{11}). Lastly, it is simple task to check that all the outcomes of this work such as well-posedness, asymptotic convergence and logarithmic decay remain valid with dynamical boundary conditions (\ref{2n}).
\label{rem1}
\end{rema}

\section{Conclusion}

This paper was concerned with the asymptotic behavior of a wave equation under the presence of a boundary delay term but at the same time without the presence of any position term. It has been shown that solutions of the system exist and are unique in an appropriate functional space. Then, applying LaSalle's invariance principle are proved to converge asymptotically to a well-defined equilibrium state. Additionally, combining  an interpolation inequality combined with a resolvent method, the convergence rate is shown to be of logarithmic type  without any extra geometric assumption in the control zone.

We would like to point out that one promising future investigation is to deal with the wave equation (\ref{1}) under the presence of a nonlinear boundary delay term such as
$$
\displaystyle \frac{\textstyle \partial y }{\textstyle \partial \nu } (x,t) = -f(y_t (x,t))-g( y_t (x,t-\tau)), \quad \mbox{on} \, \Gamma_1 \times (0,\infty),\\[1mm]
$$
where $f$ and $g$ are nonlinear functions. 

\medskip

Furthermore, it would be interesting to consider (\ref{1})-(\ref{2}) in the case of time-dependent delay $\tau (t)$. 
This will be the focus of our attention in future.

\end{document}